\documentclass[11pt,oneside,reqno]{amsart}

\usepackage{amsmath,amsthm} 
\usepackage{enumerate}
\usepackage{xfrac}          
\numberwithin{equation}{section}
\theoremstyle{definition}
\newtheorem{Definition}{Definition}[section]
\newtheorem{Example}[Definition]{Example}
\newtheorem{Remark}[Definition]{Remark}

\theoremstyle{plain}
\newtheorem{Theorem}[Definition]{Theorem}
\newtheorem*{MainTheorem}{Main Theorem}
\newtheorem{Proposition}[Definition]{Proposition}
\newtheorem{Corollary}[Definition]{Corollary}
\newtheorem{Lemma}[Definition]{Lemma}
\newtheorem{Problem}[Definition]{Problem}

\usepackage{amssymb}        
\usepackage{mathrsfs}       
\usepackage{eucal}          
\usepackage{stmaryrd}       

\usepackage{geometry}

\usepackage{hyperref}

\usepackage[usenames,dvipsnames]{xcolor}
\usepackage{tikz}
\usetikzlibrary{arrows, matrix}
\usepackage{tikz-cd}
\usepackage{subfig}         
\usepackage{arydshln}       

\newcommand{\al}{\alpha}

\newcommand{\ga}{\gamma}

\newcommand{\si}{\sigma}

\newcommand{\N}{\mathbb{N}}
\newcommand{\Z}{\mathbb{Z}}
\newcommand{\Q}{\mathbb{Q}}

\newcommand{\C}{\mathbb{C}}
\newcommand{\K}{\Bbbk}

\newcommand{\Be}{\boldsymbol{e}}

\newcommand{\CA}{\mathcal{A}}

\newcommand{\CC}{\mathcal{C}}
\newcommand{\CI}{\mathcal{I}}

\newcommand{\CL}{\mathcal{L}}

\newcommand{\CO}{\mathcal{O}}

\newcommand{\op}{\operatorname}

\DeclareMathOperator{\Aut}{Aut}

\DeclareMathOperator{\Irr}{Irr}

\DeclareMathOperator{\GL}{GL}

\newcommand{\iv}[2]{\llbracket #1,#2 \rrbracket}

\renewcommand{\tilde}{\widetilde}
\renewcommand{\mod}[1]{\;\text{(mod $#1$)}}

\title{Classification of twisted generalized Weyl algebras over polynomial rings}
\author{Jonas T. Hartwig}
\author{Daniele Rosso}
\date{\today}

\address{Department of Mathematics, Iowa State University, Ames, IA-50011, USA}
\email{jth@iastate.edu}
\address{Department of Mathematics and Actuarial Science, Indiana University Northwest, Gary, IN-46408, USA}
\email{drosso@iu.edu}

\begin{document}
\maketitle
\begin{abstract}
Let $R$ be a polynomial ring in $m$ variables over a field of characteristic zero. We classify all rank $n$ twisted generalized Weyl algebras over $R$, up to $\mathbb{Z}^n$-graded isomorphisms, in terms of higher spin 6-vertex configurations. 
Examples of such algebras include infinite-dimensional primitive quotients of $U(\mathfrak{g})$ where $\mathfrak{g}=\mathfrak{gl}_n$, $\mathfrak{sl}_n$, or $\mathfrak{sp}_{2n}$, algebras related to $U(\widehat{\mathfrak{sl}}_2)$ and a finite W-algebra associated to $\mathfrak{sl}_4$.
To accomplish this classification we first show that the problem is equivalent to classifying solutions to the binary and ternary consistency equations. Secondly, we show that the latter problem can be reduced to the case $n=2$, which can be solved using methods from previous work by the authors \cite{HarRos2016}, \cite{Har2018}.
As a consequence we obtain the surprising fact that (in the setting of the present paper) the ternary consistency relation follows from the binary consistency relation.
\end{abstract}


\section{Introduction and motivation}
\emph{Generalized Weyl algebras} (GWAs) form a class of noncommutative rings introduced by Bavula \cite{Bav1992} and, under the name \emph{hyperbolic rings}, by Rosenberg \cite{Ros1995}. They have been subject to intense research in the last few decades \cite{Bav1992,BavJor2000,BenOnd2009,Brz2016} due to their many pleasant properties. Many rings of interest in ring theory (Weyl algebras, ambiskew polynomial rings \cite{Jor1993}, generalized down-up algebras \cite{CasShe2004}), representation theory ($U(\mathfrak{g})$ and $U_q(\mathfrak{g})$ where 
$\mathfrak{g}=\mathfrak{sl}_2$, $\mathfrak{gl}_2$, or the positive part of $\mathfrak{sl}_3$) and noncommutative geometry (quantum spheres, quantum lens spaces \cite{Brz2016} and references therein) are examples of GWAs. In addition the class is closed under taking tensor products, and taking the ring of invariants with respect to finite order graded automorphisms.

When algebraic structures are given by presentations, the question arises whether two different presentations give rise to isomorphic algebras. This is known as the isomorphism problem. In general, answering this question can be very difficult. For example, to determine whether two presented groups are isomorphic is known to be an NP hard problem.
The isomorphism problem for GWAs has been studied in \cite{BavJor2000,CarLop2009,Gad2014,SuaViv2015,Tan2018,KitLau2019}.

Despite the variety of interesting examples of GWAs, there are some algebras that ``should'' be GWAs but are not, for example multiparameter quantized Weyl algebras defined by Maltsiniotis \cite{Mal1990} and studied in \cite{Jor2011}. To remedy this deficiency, as well as including examples related to higher rank Lie algebras, Mazorchuk and Turowska introduced the class of \emph{twisted generalized Weyl algebras} (TGWAs). Their structure and representation theory have been extensively studied. See for example \cite{MazTur1999,MazTur2002,MazPonTur2003,Ser2001,Har2006,Har2009,HarOin2013,FutHar2012a,FutHar2012b,HarRos2016,Har2018,Har2017} and references therein.
Besides quantized Weyl algebras, examples include quotients of enveloping algebras of simple Lie algebras by annihilators of completely pointed (i.e. multiplicity-free) simple weight modules \cite{HarSer2016}.

An additional difficulty arise for TGWAs, as they are not in general free as left modules over their degree zero subring. In particular, for some choices of input data the relations will be contradictory, i.e. the algebra will be trivial, consisting of a single element $0=1$, see \cite{FutHar2012a}.

To show that the algebra is nontrivial, the strategy was to show that the degree zero subring is isomorphic to the base ring $R$. This question can be seen as a weak isomorphism problem, because we are asking for the structure of the degree zero subring.
The full isomorphism problem for TGWAs has only been studied for special classes such as multiparameter quantized Weyl algebras \cite{GooHar2015,Gad2017,Tan2016}.

In this paper we solve the graded isomorphism problem for a natural class of TGWAs.
First we show that the graded isomorphism problem is equivalent to classifying solutions to the binary and ternary consistency equations, up to a natural equivalence that we define.
The latter amounts to the following polynomial problem, which is a higher rank and multivariate generalization of the problem addressed in \cite{HarRos2016}.

\begin{Problem}
Let $\K$ be a field of characteristic zero, $m$ a positive integer and let 
 $R=\K[u_1,u_2,\ldots,u_m]$ 
 be the polynomial algebra over $\K$ in $m$ commuting independent indeterminates $u_j$.
Find all pairs $(\alpha,p)$, where $\alpha=(\al_{ij})\in M_{m\times n}(\K)$ is a matrix and $p=(p_1,p_2,\ldots,p_n)\in R^n$ is an $n$-tuple of polynomials satisfying the following equations:
\begin{subequations}\label{eq:consistency-eqs}
\begin{align}
\label{eq:binary}
p_i(u-\tfrac{\al_j}{2})p_j(u-\tfrac{\al_i}{2})&=p_i(u+\tfrac{\al_j}{2})p_j(u+\tfrac{\al_i}{2}) \quad \forall i\neq j\\
\label{eq:ternary}
p_k(u-\tfrac{\al_i}{2}-\tfrac{\al_j}{2})p_k(u+\tfrac{\al_i}{2}+\tfrac{\al_j}{2})&=p_k(u-\tfrac{\al_i}{2}+\tfrac{\al_j}{2})p_k(u+\tfrac{\al_i}{2}-\tfrac{\al_j}{2})\quad \forall i\neq j\neq k\neq i
\end{align}
\end{subequations}
where $u=(u_1,u_2,\ldots,u_m)$ and $\al_j=(\al_{1j},\al_{2j},\ldots,\al_{mj})$ for $j=1,2,\ldots,n$.
\end{Problem}

In \cite{HarRos2016}, we solved this problem in the case of $(m,n)=(1,2)$ and $\K=\C$.

Secondly, generalizing results from \cite{HarRos2016}, we prove that any solution can be factored into \emph{orbital} solutions, i.e. solutions all of whose irreducible factors belong to a given $\Z^n$-orbit in $R$. (In \cite{HarRos2016}, where $(m,n)=(1,2)$, orbital solutions were called \emph{integral solutions}.)
Next we prove that any orbital solution is essentially rank two (see Corollary \ref{cor:nontriv}). 
Finally, we extend the methods of \cite{HarRos2016,Har2018} to give a combinatorial classification of all essentially rank two solutions in terms of vertex configurations on a square lattice.

As a surprising consequence, we deduce that, when $R$ is a polynomial ring and the automorphisms are given by additive shifts, the binary consistency relation actually implies the ternary consistency relation. In general, as shown in \cite{FutHar2012a}, these relations are independent.

\subsection*{Future directions}
It is natural to ask what will happen in finite characteristic. Some solutions to the consistency equations in finite characteristic were given in \cite{MazTur1999}. Another direction would be to study this question when the base algebra is not a polynomial algebra, for example a Laurent polynomial ring. This case is expected to be related to quantum analogs of the algebras from this paper. Lastly we remark that the problem of determining consistency equations necessary and sufficient for a given TGWA to be consistent is still open in the case when certain elements $t_i$ (see Definition \ref{def:TGW-stuff}) are allowed to be zero divisors. Such TGWAs appear naturally as matrix algebras that are finite-dimensional primitive quotients 
of enveloping algebras \cite{HarSer2016}, or as algebras of differential operators with $\mathcal{C}^{\infty}$-function coefficients.

\section{On the graded isomorphism problem for twisted generalized Weyl algebras}

\subsection{Definition of twisted generalized Weyl algebra}

We use slightly different notation from other papers in  the literature.
Specifically, our use of $\si_i^{1/2}$ comes from the fact that we want to write the equations \eqref{eq:consistency-eqs-sigma} in a symmetric way.
This system of notation is equivalent to the original one from \cite{MazTur1999} by simple substitutions, see Section \ref{sec:symmetrized} for details.

\begin{Definition} \label{def:TGW-stuff}
Let $n$ be a positive integer.
\begin{enumerate}[{\rm (i)}]
\item A \emph{twisted generalized Weyl datum (TGWD) of rank $n$} is a triple $(R,\si,t)$ where $R$ is a ring, $\si=(\si_1^{1/2},\si_2^{1/2},\ldots,\si_n^{1/2})$ is an $n$-tuple of commuting ring automorphisms of $R$, and $t=(t_1,t_2,\ldots,t_n)$ is an $n$-tuple of nonzero elements from the center of $R$.
\item $(R,\si,t)$ is \emph{regular} if $t_i$ is regular (i.e. not a zero-divisor) in $R$ for all $i$.
\item $(R,\si,t)$ is \emph{consistent} if the following equations hold:
\begin{subequations}\label{eq:consistency-eqs-sigma}
\begin{align}
\label{eq:binary-consistency-eq-sigma}
\si_j^{1/2}(t_i) \cdot \si_i^{1/2}(t_j) &= 
 \si_j^{-1/2}(t_i) \cdot \si_i^{-1/2}(t_j)\;\; (i\neq j)\\
\si_i^{1/2}\si_j^{1/2}(t_k) \cdot 
\si_i^{-1/2}\si_j^{-1/2}(t_k) &=
\si_i^{1/2}\si_j^{-1/2}(t_k) \cdot
\si_i^{-1/2}\si_j^{1/2}(t_k)\;\; (i\neq j\neq k\neq i)
\end{align}
\end{subequations}
\item Let $(R,\si,t)$ be a TGWD of rank $n$. The associated \emph{twisted generalized Weyl construction (TGWC)}, $\CC(R,\si,t)$, is the $\Z^n$-graded $R$-ring generated by $2n$ indeterminates $X_1^+,X_1^-,\ldots,X_n^+,X_n^-$ subject to:
\begin{subequations}\label{eq:tgwa-rels}
\begin{align}
\label{eq:tgwa-rel-1}
X_i^\pm X_i^\mp &=\si_i^{\pm 1/2}(t_i)1,  \\
\label{eq:tgwa-rel-2}
X_i^\pm X_j^\mp &=X_j^\mp X_i^\pm \;\; (i\neq j), \\
\label{eq:tgwa-rel-3}
X_i^\pm \si_i^{\mp 1/2}(r) &=\si_i^{\pm 1/2}(r)X_i^\pm\;\; (\forall r\in R), \\
\deg(X_i^\pm)&=\pm\Be_i, \\
\deg (r1) &=\boldsymbol{0} \;\;(\forall r\in R),
\end{align}
\end{subequations}
where $\oplus_{i=1}^n \Z\Be_i=\Z^n$.
\item Let $(R,\si,t)$ be a TGWD of rank $n$. The associated \emph{twisted generalized Weyl algebra (TGWA)}, $\CA(R,\si,t)$, is the quotient ring $\CC(R,\si,t)/\CI$ where $\CI$ is the sum of all graded ideals $J=\oplus_{\boldsymbol{d}\in\Z^n} J_{\boldsymbol{d}}$ such that $J_{\boldsymbol{0}}=\{0\}$.
\end{enumerate}
\end{Definition}

A basic problem for rings given by generators and relations is whether the ring is non-trivial. The following result, which was the main theorem from \cite{FutHar2012a}, gives a sufficient condition in the case of TGWCs and TGWAs.
This motivates the terminology in Definition \ref{def:TGW-stuff}(iii).

\begin{Theorem}[{\cite{FutHar2012a}}] \label{thm:consistency}
Let $(R,\si,t)$ be a regular TGWD of rank $n$, let $C=\CC(R,\si,t)$ and $A=\CA(R,\si,t)$ be the associated TGWC and TGWA respectively, with $\Z^n$-gradation $C=\bigoplus_{\boldsymbol{d}\in\Z^n}C_{\boldsymbol{d}}, A=\bigoplus_{\boldsymbol{d}\in\Z^n}A_{\boldsymbol{d}}$.  Then the following are equivalent:
\begin{enumerate}[{\rm (i)}]
\item $(R,\si,t)$ is consistent,
\item the canonical map $R\to C_{\boldsymbol{0}}$ is an isomorphism,
\item the canonical map $R\to A_{\boldsymbol{0}}$ is an isomorphism.
\end{enumerate}
In particular, if $(R,\si,t)$ is regular and consistent, then $C$ and $A$ are both non-trivial rings.
\end{Theorem}

That the canonical maps are both surjective is easy to see. The difficult part is to prove they are injective. The proof involves the diamond lemma and the relations \eqref{eq:consistency-eqs-sigma}. Relation \eqref{eq:binary-consistency-eq-sigma} appeared already in \cite{MazTur1999}.

We will need the following lemma, which is a special case of \cite[Lem.~3.2]{Har2006}

\begin{Lemma} \label{lem:cyclic-components}
 Let $(R,\si,t)$ be a TGWD of rank $n$. Let $C=\CC(R,\si,t)$ and $A=\CA(R,\si,t)$ be the corresponding TGWC and TGWA. Then:
\begin{equation}
C_{\pm \Be_i}=C_{\boldsymbol{0}} X_i^\pm=X_i^\pm C_{\boldsymbol{0}}
\end{equation}
\begin{equation}
A_{\pm \Be_i}=A_{\boldsymbol{0}} X_i^\pm=X_i^\pm A_{\boldsymbol{0}}
\end{equation}
\end{Lemma}

\subsection{Graded isomorphism problem}
Let $A=\CA(R,\si,t)$ and $A'=\CA(R',\si',t')$ be two TGWAs of the same rank $n$. We say that $A$ and $A'$ are \emph{graded isomorphic} if there exists a ring isomorphism $\varphi:A\to A'$ such that $\varphi(A_{\boldsymbol{d}})\subseteq A'_{\boldsymbol{d}}$ for all $\boldsymbol{d}\in\Z^n$.

Put $\si_i=(\si_i^{1/2})^2$.
In this section we show that, under certain conditions, two TGWAs are graded isomorphic if and only if their respective TGWDs $(R,\si,t)$ and $(R',\si',t')$ are equivalent in the following sense.

\begin{Definition}
Two TGWDs of rank $n$, $(R,\si,t)$ and $(R',\si',t')$, are \emph{equivalent} if there exists a ring isomorphism $\psi:R\to R'$ such that $\psi\circ \si_i=\si_i'\circ\psi$ and $\psi(t_i)\in (R')^\times \cdot t_i'$ for all $i$;
\end{Definition}

Here $R^\times$ denotes the group of units (invertible elements) in a ring $R$.
To state the result, we will need the following definitions.

\begin{Definition} Let $(R,\si,t)$ be a TGWD.
\begin{enumerate}[{\rm (i)}]
\item $(R,\si,t)$ is \emph{commutative} if the ring $R$ is commutative;
\item $(R,\si,t)$ has \emph{scalar units} if $\si_i(u)=u$ for every $u\in R^\times$  and every $i$.
\end{enumerate}
\end{Definition}

\begin{Proposition}\label{prop:gradediso}
Let $(R,\si,t)$ and $(R',\si',t')$ be two regular, consistent, commutative, TGWDs of rank $n$ having scalar units. Let $A=\CA(R,\si,t)$ and $A'=\CA(R',\si',t')$ be the corresponding TGWAs.
Then $A$ and $A'$ are graded isomorphic if and only if $(R,\si,t)$ is equivalent to $(R',\si',t')$.
\end{Proposition}

\begin{proof}
Suppose that $(R,\si,t)$ and $(R',\si',t')$ are equivalent and let $\psi:R\to R'$ be a ring isomorphism such that $\psi\circ\si_i=\si_i'\circ\psi$ and $\psi(t_i)= r_i\cdot t_i'$ for all $i$, where $r_i\in R^\times$. Define $\Psi(X_i^+)={X_i'}^+$, $\Psi(X_i^-)=r_i {X_i'}^-$, and $\Psi(r)=\psi(r)$ for all $r\in R$. It is straightforward to verify that Relations \eqref{eq:tgwa-rels} are preserved by $\Psi$. Thus $\Psi$ extends to a ring homomorphism $\Psi:\CC(R,\si,t)\to\CC(R',\si',t')$. Clearly $\Psi$ is a graded homomorphism. Replacing $\psi$ by $\psi^{-1}$ we obtain a graded homomorphism which is the inverse of $\Psi$. Thus $\Psi$ is an isomorphism of $\Z^n$-graded rings. Since $\Psi$ is graded we have $\Psi(\CI)\subseteq \CI'$ where $\CI$ and $\CI'$ are as in the definition of TGWA. Thus $\Psi$ induces a graded isomorphism of the quotients $\CA(R,\si,t)\to\CA(R',\si',t')$.

Conversely, suppose $\Psi:A\to A'$ is a graded isomorphism.
That is, $\Psi$ is a ring isomorphism such that 
$\Psi(A_{\boldsymbol{d}})=(A')_{\boldsymbol{d}}$ for every $\boldsymbol{d}\in\Z^n$.
By Theorem \ref{thm:consistency} we can identify $A_{\boldsymbol{0}}\cong R$ and $(A')_{\boldsymbol{0}}\cong R'$. Thus taking $\boldsymbol{d}=\boldsymbol{0}$ we get a ring isomorphism $\psi=\Psi|_{R}:R\to R'$. To show that $\psi$ satisfies the required properties we also need to consider the case $\boldsymbol{d}=\pm \Be_i$.
%
%
By Lemma \ref{lem:cyclic-components},
$A_{\pm \Be_i}=RX_i^\pm$ and $A'_{\pm \Be_i}=R'{X'_i}^\pm$
(denoting the generators of $A'$ by ${X'_i}^\pm$).
From $\Psi(A_{\pm\Be_i})=(A')_{\pm\Be_i}$ we conclude that $\Psi(X_i^\pm)={r_i'}^\pm  {X_i'}^\pm $ for some ${r_i'}^\pm \in R'$, and $\Psi^{-1}({X_i'}^\pm) = r_i^\pm X_i^+$ for some $r_i^\pm\in R$. Combining these equalities we get
\[
X_i^\pm = \Psi^{-1}\big(\Psi(X_i^\pm)\big) = \psi^{-1}({r_i'}^\pm ) r_i^\pm  X_i^\pm.
\]
Multiplying from the right by $X_i^\mp$ and using \eqref{eq:tgwa-rel-1} we get
\[
\si_i^{\pm 1/2}(t_i) = \psi^{-1}({r_i'}^\pm) r_i^\pm  \si_i^{\pm 1/2}(t_i)
\]
hence
\[
\si_i^{\pm 1/2}(t_i) \cdot (1 - \psi^{-1}({r_i'}^\pm) r_i^\pm  ) = 0
\]
Since $t_i$ is regular in $R$, the same is true for $\si_i^{\pm 1/2}(t_i)$. Hence
\[
\psi^{-1}({r_i'}^\pm ) r_i^\pm  = 1
\]
i.e. $r_i^\pm $ and ${r_i'}^\pm $ are units.



Next, applying $\Psi$ to the relation $X_i^\pm \si_i^{\mp 1/2}(r)= \si_i^{\pm 1/2}(r) X_i^\pm$ we obtain:
\begin{equation}
{r_i'}^\pm {X'_i}^\pm \psi\big(\si_i^{\mp 1/2}(r)\big) = \psi\big(\si_i^{\pm 1/2}(r)\big) {r_i'}^\pm {X'_i}^\pm
\end{equation}
Dividing both sides by the unit ${r_i'}^\pm$ and 
multiplying from the right by ${X'}_i^\mp$ and using \eqref{eq:tgwa-rel-1} and \eqref{eq:tgwa-rel-3}, we get
\begin{equation}
{\si'_i}^{\pm 1/2}(t_i') \cdot {\si'_i}^{\pm 1}\circ \psi\circ \si_i^{\mp 1/2} (r)
=
\psi(\si_i^{\pm 1/2}(r)\big) {\si'_i}^{\pm 1/2}(t_i')
\end{equation}
Since $t_i'$ is central and not a zero-divisor, the same is true for ${\si_i'}^{\pm 1/2}(t_i')$ and we may cancel it from both sides. Furthermore since the equation holds for all $r\in R$ we obtain
\begin{equation}
{\si_i'}^{\pm 1}\circ \psi = \psi \circ \si_i^{\pm 1}
\end{equation}
for all $i$.
Therefore
\begin{align*}
\psi(t_i) &={\si_i'}^{\pm 1/2}\circ\psi\circ \si_i^{\pm 1/2}(t_i) \\
&={\si_i'}^{\mp 1/2}\circ\Psi(X_i^\pm X_i^\mp) \\
&={\si_i'}^{\mp 1/2}\big({r_i'}^\pm {X_i'}^\pm {r_i'}^{\mp} {X_i'}^\mp) \\
&={\si_i'}^{\mp 1/2}\big({r_i'}^\pm {\si_i'}^{\pm 1}({r_i'}^{\mp})\cdot {\si_i'}^{\pm 1/2}(t_i')\big) \\
&= {\si_i'}^{\mp 1/2}\big({r_i'}^\pm {\si_i'}^{\pm 1}({r_i'}^{\mp})\big)\cdot t_i' \\
&\in (R')^\times\cdot t_i'
\end{align*}
This proves that $\psi$ gives an equivalence between $(R,\si,t)$ and $(R',\si',t')$.
\end{proof}

\begin{Remark}
Notice that, as a consequence, if we have two TGWAs satisfying the hypotheses of the theorem, defined respectively over the rings $R$ and $R'$, that are graded isomorphic, then the rings are isomorphic, so we can identify $R$ and $R'$.
\end{Remark}
Our goal is to classify TGWAs defined over a polynomial ring, up to graded isomorphism, so we introduce the following definition.
\begin{Definition}\label{def:polyequiv}
Let $\K$ be a field of characteristic zero, and $R=\K[u_1,u_2,\ldots,u_m]$. Let $p=(p_1,\ldots,p_n)$ (resp. $p'=(p_1',\ldots,p_n')$) be an $n$-tuple of nonzero monic elements of $R$, and $\alpha=(\al_{ij})\in M_{m\times n}(\K)$ (resp. $\alpha'=(\al'_{ij})\in M_{m'\times n}(\K)$) and $\si_i\in\Aut_\K(R)$ given by $\si_i(u_j)=u_j-\al_{ji}$ (resp. $\si'_i\in\Aut_\K(R')$, given by $\si'_i(u_j)=u'_j-\al'_{ji}$). We say that the pairs $(\alpha,p)$ and $(\alpha',p')$ are \emph{equivalent} is there exists an automorphism $\psi\in\Aut(R)$ such that $\psi\circ\si_i=\si'_i\circ\psi$ and $\psi(p_i)\in \K^\times \cdot p_i'$ for all $1\leq i\leq n$.
\end{Definition}
\begin{Example}\label{exa:GLm}
Let $g\in\GL_m(\K)=\Aut(\oplus_{i=1}^m \K u_i)$ be an invertible $m\times m$ matrix, then $g$ gives an automorphism $\psi_g\in\Aut(R)$ by acting linearly on the variables. In this case, we have that if $\psi_g\circ\si_i=\si'_i\circ\psi_g$, then $\alpha'=g\alpha$ and $p'=(p_1',\ldots, p_n')=(p_1\circ g,\ldots, p_n\circ g)$. This shows that there is a $\GL_m(\K)$-action on the set of pairs $(\alpha,p)$ as defined in Def. \ref{def:polyequiv}, and that $\GL_m(\K)$-orbits are contained in the equivalence classes.
\end{Example}
\begin{Example}
Let $\alpha=0$, and $\psi\in\Aut(R)$ be any automorphism of the polynomial ring, then $\psi\circ\si_i=\si'_i\circ\psi$ is always satisfied, so $(0,p)$ and $(0,\psi(p))$ are always equivalent no matter what $\psi$ is. For example, we could choose $\psi\in\Aut(R)$ defined by $\psi(u_j)=u_j+q_j(u_{j+1},\ldots,u_m)$ for some polynomials $q_j$, $1\leq j\leq m$. In particular, this shows that the equivalence classes of pairs $(\alpha,p)$ are bigger in general than the $\GL_m(\K)$-orbits described in Example \ref{exa:GLm}.
\end{Example}
In the case of polynomial rings $R$, Proposition \ref{prop:gradediso} gives us the following.
\begin{Corollary}
Let $R,\si,\si',p,p'$ as in Definition \ref{def:polyequiv} such that $(R,\sigma,p)$ and $(R,\sigma',p')$ are consistent TGWDs.

Then the TGWAs $A=\CA(R,\si,p)$ and $A'=\CA(R,\si',p')$ are isomorphic as graded $\K$-algebras if and only if $(\al,p)$ and $(\al',p')$ are equivalent.
\end{Corollary}
\begin{proof}
This follows directly from Proposition \ref{prop:gradediso} with the observation that, given our setup, the TGWDs $(R,\si,p)$ and $(R,\si',p')$ are also regular, commutative and have scalar units. In addition, the automorphism $\psi$ giving the equivalence has to be $\K$-linear.
\end{proof}
%

\section{Orbital solutions}
In this section, unless otherwise stated, $\K$ can be any field with $\op{char}\K\neq 2$.
Equipping the gradation monoid $\N^m$ with an order, such as lexicographical, each nonzero element of $R$ has a unique leading monomial. We say an element of $R$ is monic if its leading monomial has coefficient $1$. Note that the set of monic elements of $R$ is invariant under the automorphisms $\si_i$. We say that $p=(p_1,\ldots,p_n)\in R^n$ is monic if $p_i$ is (nonzero and) monic for each $i$.

Let $\Irr(R)$ denote the set of monic irreducible polynomials in $R$. The commuting automorphisms $\si_i$ induce an action of the group $\Z^n$ on $\Irr(R)$ via 
\begin{equation}\label{eq:Zn-action}
(a_1,a_2,\ldots,a_n).p = \si_1^{a_1}\circ\si_2^{a_2}\circ\cdots\circ\si_n^{a_n}(p)\end{equation}
Since $\operatorname{char} \K\neq 2$, $\si_i$ have square roots in $\Aut_\K(R)$ given by $\si_i^{1/2}(u_j)=u_j-\frac{1}{2}\al_{ji}$.

\begin{Definition}
Let $\CO\in\Irr(R)/\Z^n$ be an orbit of monic irreducible polynomials with respect to the $\Z^n$ action defined in \eqref{eq:Zn-action}.
A monic element $p=(p_1,p_2,\ldots,p_n)\in R^n$ is \emph{$\CO$-orbital} if for each $i\in\{1,2,\ldots,n\}$, every irreducible factor of $p_i$ belongs $\si_i^{1/2}(\CO)=\big\{\si_i^{1/2}(q)\mid q\in\CO\big\}$.
\end{Definition}

%

\begin{Theorem}\label{theorem:integral-product}
Let $p$	be any monic solution to \eqref{eq:consistency-eqs}. Then there exists a unique finite subset 
\begin{equation}
\big\{(\CO_1,p^{(1)}),(\CO_2,p^{(2)}),\ldots,(\CO_k,p^{(k)})\big\}\subseteq (\Irr(R)/\Z^n) \times R^n
\end{equation}
such that
\begin{enumerate}[{\rm (i)}]
\item $p=p^{(1)}p^{(2)}\cdots p^{(k)}$ as an equality in the direct product ring $R^n$,
\item $p^{(i)}$ is a monic solution to \eqref{eq:consistency-eqs} for each $i\in\{1,2,\ldots,k\}$,
\item $p^{(i)}$ is $\CO_i$-orbital for each $i\in\{1,2,\ldots,k\}$,
\item $\CO_i\neq\CO_j$ for all $i\neq j$.
\end{enumerate}
\end{Theorem}

\begin{proof}
Let $\CO\in\Irr(R)/\Z^n$ be any orbit. Let $p^{\CO}=(p^{\CO}_1,p^{\CO}_2,\ldots,p^{\CO}_n)\in R^n$, where $p^{\CO}_i$ is the product of all monic irreducible factors of $p_i$ belonging to $\si_i^{1/2}(\CO)$. By convention $p_i^{\CO}=1$ if there are no such irreducible factors. Then clearly $p=\prod p^{\CO}$ where $\CO$ ranges over all orbits in $\Irr(R)/\Z^n$. Moreover each $p^{\CO}$ is $\CO$-orbital. So it remains to prove that $p^{\CO}$ is a solution to \eqref{eq:consistency-eqs} for each $\CO$. 

Fix $i,j\in\{1,2,\ldots,n\}, i\neq j$.
Let $f$ be any irreducible factor of $p_i^{\CO}(u-\al_j/2)p_j^\CO(u-\al_i/2)$. Without loss of generality suppose that $f$ divides $p_i^{\CO}(u-\al_j/2)$. Then $\si_j^{-1/2}(f)$ divides $p_i^{\CO}(u)$. Hence $f$ belongs to $\si_i^{1/2}\si_j^{1/2}(\CO)$. On the other hand $f$ divides $p_i(u-\al_j/2)$, and hence since $p$ solves \eqref{eq:consistency-eqs}, $f$ divides $p_i(u+\al_j/2)p_j(u+\al_i/2)$. Since $f$ is irreducible, $f$ divides either $p_i(u+\al_j/2)$ or $p_j(u+\al_i/2)$. If $f$ divides $p_i(u+\al_j/2)$ then $\si_j^{1/2}(f)$ divides $p_i(u)$, hence $f$ divides $p_i^\CO(u+\al_j/2)$. Similarly in the other case $f$ divides $p_j^\CO(u+\al_i/2)$. This shows that any irreducible factor in the left hand side of \eqref{eq:binary} for $p^\CO$ divides the right hand side. Symmetrically the right hand side divides the left hand side. Thus $p^\CO$ solve \eqref{eq:binary}.

An analogous argument shows that $p^\CO$ also satisfies the relation \eqref{eq:ternary}.
\end{proof}

\section{Reduction to Rank Two}
In this section, unless otherwise stated, $\K$ is an arbitrary field of characteristic not two.
\begin{Proposition}\label{prop:indices-ij}
Let $p$ be a monic $\CO$-orbital solution to \eqref{eq:binary}, where $\CO=\Z^n q_0$ ($q_0\in\Irr(R)$) and let $i,j\in\{1,\ldots,n\}$, $i\neq j$, such that $p_i$ and $p_j$ are not constant. Then there exists a pair of nonnegative integers $(r_{ij},s_{ij})\neq (0,0)$ such that
\begin{equation}
\label{eq:fixq0} \si_i^{r_{ij}}\si_j^{s_{ij}}q_0=q_0.
\end{equation} 
\end{Proposition}

\begin{proof}
Let $f$ be an irreducible factor of $p_i$. Then $\si_j^{1/2}f$ divides the LHS of \eqref{eq:binary}, which means that $\si_ j^{1/2}f$ divides the RHS of the same equation. We have then that $\si_j^{1/2}f| \si_j^{-1/2}p_i$ or $\si_j^{1/2}f|\si_i^{-1/2}p_j$. Say that $\si_j^{1/2}f| \si_j^{-1/2}p_i$, then $\si_jf|p_i$ (in the other case we get that $\si_i^{1/2}\si_j^{1/2}f|p_j)$. Iterating the argument, we have a sequence $\{\si_i^{r_k}\si_j^{s_k}f\}$ of irreducible factors of $p_i$, and a sequence $\{\si_i^{r_\ell+1/2}\si_j^{s_\ell+1/2}f\}$ of irreducible factors of $p_j$, with $r_k$, $s_k$, $r_\ell$, $s_\ell$ nonnegative integers. Since $R$ is a UFD, the sequence $\{\si_i^{r_k}\si_j^{s_k}f~:~k\geq 1\}$ will have repetitions, which implies that there is a pair of nonnegative integers $(r,s)\neq (0,0)$ such that $\si_i^r \si_j^s f=f$. This argument can be repeated for all irreducible factors of $p_i$ and $p_j$ and, by taking the least common multiple of the powers, we have that for each pair $(i,j)$, $i,j\in\{1,\ldots,n\}$, $i\neq j$ there is a pair of nonnegative integers $(r_{ij},s_{ij})\neq (0,0)$ such that $\si_i^{r_{ij}}\si_j^{s_{ij}}$ acts as the identity on all the irreducible factors of $p_i$ and $p_j$. Since $\CO=\Z^nq_0$, with $q_0\in \Irr(R)$, any irreducible factor $q_i$ of $p_i$ is in $\si_i^{1/2}\CO$, so it is of the form $q_i=\si_i^{1/2}\si_1^{a_1}\cdots\si_n^{a_n}q_0$. Then
\begin{align}
\notag \si_i^{r_{ij}}\si_j^{s_{ij}} q_i&= q_i \\
\notag \si_i^{r_{ij}}\si_j^{s_{ij}}\si_i^{1/2}\si_1^{a_1}\cdots\si_n^{a_n}q_0&=\si_i^{1/2}\si_1^{a_1}\cdots\si_n^{a_n}q_0 \\
\notag \si_i^{1/2}\si_1^{a_1}\cdots\si_n^{a_n}\si_i^{r_{ij}}\si_j^{s_{ij}}q_0&=\si_i^{1/2}\si_1^{a_1}\cdots\si_n^{a_n}q_0 \\
\notag \si_i^{r_{ij}}\si_j^{s_{ij}}q_0&=q_0.
\end{align}
\end{proof}

\begin{Proposition}\label{prop:almost-trivial}
Suppose $\op{char}\K=0$.
Let $p$ be a monic $\CO$-orbital solution to \eqref{eq:binary}, with $\CO=\Z^n q_0$.
Let $i\neq j$ be two indices such that $p_i=1$ and $p_j\neq 1$, then $\si_i q_0=q_0$.
Viceversa, let $i\neq j$ be two indices such that $\si_i q_0=q_0$ and $\si_j q_0\neq q_0$, then $p_i=1$.
\end{Proposition}
\begin{proof}
If there are two indices such that $p_i=1$ and $p_j\neq 1$, then by \eqref{eq:binary} we get
\begin{align*}
\si_i^{1/2}p_j\si_j^{1/2}p_i&=\si_i^{-1/2}p_j\si_j^{-1/2}p_i\\
\si_i^{1/2}p_j&=\si_i^{-1/2}p_j \\
\si_i p_j=p_j.
\end{align*}
Since $p_j$ is non trivial, this implies that $\si_i q_0=q_0$.

Viceversa, suppose that $\si_i q_0=q_0$, and that there is an index $j$ such that $\si_j q_0\neq q_0$, then by \eqref{eq:binary} we get
\begin{align*}
\si_i^{1/2}p_j\si_j^{1/2}p_i&=\si_i^{-1/2}p_j\si_j^{-1/2}p_i\\
p_j\si_j^{1/2}p_i&=p_j\si_j^{-1/2}p_i \\
\si_j p_i=p_i.
\end{align*}
Since $p_i$ is a product of shifts of $q_0$, and $\si_j q_0\neq q_0$, then the only possiblity is that $p_i=1$, because $\op{char}\K=0$.
\end{proof}

\begin{Lemma}\label{Lemma:fix-auto}
Suppose $\op{char}\K=0$, and let $\si(u)=u-\beta$, that is $\si(u_1, \ldots, u_m)=(u_1-\beta_1, \ldots, u_m-\beta_m)$. For $q_0\in\Irr(R)$
$$\si q_0=q_0\quad\text{ if and only if }\quad \langle \nabla q_0, \beta\rangle\equiv0.$$
Here $\nabla q_0=\left( \frac{\partial q_0}{\partial u_1},\ldots,\frac{\partial q_0}{\partial u_m}\right)$ is the gradient and 
$$ \langle~,~\rangle: R^m\times \K^m\to R^m, \qquad \langle (\gamma_k), (\beta_k)\rangle = \sum_{k=1}^m\gamma_k\beta_k.$$
\end{Lemma}
\begin{proof}
The irreducible polynomial $q_0$ is invariant under the automorphism $\si$ if and only if $q_0(u-\beta)=q_0(u)$, which is equivalent to $q_0(u+r\beta)=q_0(u)$ for all $r\in\Z$. Since $q_0$ is a polynomial and $\op{char}\K=0$, this is the same as $q_0(u+t\beta)=q_0(u)$ for all $t\in\K$, which is indeed equivalent to $\langle \nabla q_0, \beta\rangle\equiv 0$ (as can be easily seen by considering $q_0(u+t\beta)$ as a polynomial function of the variable $t$).
\end{proof}

\begin{Theorem}\label{theorem:less3}
Let $\op{char}\K=0$ and let $p$ be a monic $\CO$-orbital solution to \eqref{eq:binary}, with $\CO=\Z^n q_0$. Then there are at most two distinct indices $1\leq i\leq n$ such that $p_i$ is not constant and $\si_ i q_0\neq q_0$.
\end{Theorem}

\begin{proof}
By Proposition \ref{prop:indices-ij}, for each pair of indices $i$, $j$ such that $p_i$ and $p_j$ are not constant, we have \eqref{eq:fixq0}. By Lemma \eqref{Lemma:fix-auto} we have that
$$  \si_i^{r_{ij}}\si_j^{s_{ij}}q_0(u)=q_0(u-r_{ij}\al_i-s_{ij}\al_j)=q_0(u_1-r_{ij}\al_{1i}-s_{ij}\al_{1j},\ldots,u_n-r_{ij}\al_{ni}-s_{ij}\al_{nj}) $$
is equal to $q_0(u)$ if and only if
\begin{align}\notag \langle \nabla q_0, r_{ij}\al_{i}+s_{ij}\al_{j}\rangle&=0 \\
\label{eq:linear} r_{ij}\langle \nabla q_0, \al_i\rangle +s_{ij} \langle \nabla q_0, \al_j \rangle &= 0
\end{align}

Now suppose that $i\neq j$ are such that $\si_i q_0\neq q_0\neq \si_j q_0$, i.e., $\langle \nabla q_0, \al_i\rangle\neq0\neq\langle \nabla q_0, \al_j\rangle$. Then we have $r_{ij}, s_{ij}>0$ and \eqref{eq:linear} becomes
\begin{align}\notag r_{ij}\langle \nabla q_0, \al_i\rangle +s_{ij} \langle \nabla q_0, \al_j \rangle &= 0 \\
\notag r_{ij}\langle \nabla q_0, \al_i\rangle &=-s_{ij} \langle \nabla q_0, \al_j \rangle \\
\label{eq:cneg-ratio}\langle \nabla q_0, \al_i\rangle &=-\frac{s_{ij}}{r_{ij}} \langle \nabla q_0, \al_j\rangle.
\end{align}
 

Now suppose that $i,j,k$ are three distinct indices as in the statement of the Theorem, then by \eqref{eq:cneg-ratio}, we have negative rational numbers $r_1$, $r_2$ and $r_3$ such that
\begin{align*}\langle \nabla q_0, \al_i\rangle&=r_1 \langle \nabla q_0, \al_j\rangle \\
&= r_1 (r_2\langle \nabla q_0, \al_k\rangle) \\
&= r_1 r_2 (r_3\langle \nabla q_0, \al_i\rangle) \\
&= (r_1r_2r_3)\langle \nabla q_0, \al_i\rangle
\end{align*}
which is impossible.




\end{proof}

Notice that if there is a single index $i$ such that for all other indices $j\neq i$ we have $p_j=1$ and $\si_jq_0=q_0$, then for any value of $\al_i$ and any choice of $p_i$ we trivially get a solution.

As a consequence of Proposition \ref{prop:almost-trivial} and Theorem \ref{theorem:less3}, we have the following immediate consequence.

\begin{Corollary}\label{cor:nontriv}
Suppose that $\op{char} \K= 0$.
For $\CO=\Z^nq_0$, any monic $\CO$-orbital nontrivial solutions $p$ of \eqref{eq:binary} satisfies the following: there are two indices $i,j\in \{1,\ldots, n\}$, $i\neq j$ such that $p_i\neq 1\neq p_j$, $\si_i q_0\neq q_0\neq \si_j q_0$ and for all $k\in\{1,\ldots, n\}\setminus\{i,j\}$ we have $p_k=1$ and $\si_kq_0=q_0$.
\end{Corollary}

We can then also conclude as follows.

\begin{Corollary}\label{cor:binary-implies-ternary}
For $R=\K[u_1,\ldots,u_m]$, $\op{char}\K=0$ and $\si_i=u-\al_i$, any monic solution $p$ to \eqref{eq:binary} also satisfies \eqref{eq:ternary}.
\end{Corollary}

\begin{proof}
If $p$ is a solution to \eqref{eq:binary}, then by Theorem \ref{theorem:integral-product} $p=p^{(1)}\cdots p^{(k)}$ with $p^{(i)}$ being a monic $\CO_i$-orbital solution. By Corollary \ref{cor:nontriv}, each $p^{(i)}$ is either a trivial solution or it is only nontrivial in two components. It follows immediately that for all $1\leq i\leq k$, $p^{(i)}$ trivially satisfies \eqref{eq:ternary}, hence so does $p$.
\end{proof}

\section{Rank two solutions via higher spin $6$-vertex configurations}

In this section we generalize the rank two solutions obtained in \cite{HarRos2016} from univariate complex polynomials to multivariate polynomials. In Theorem \ref{thm:vertex}, $\K$ can be an arbitrary field of characteristic not equal to two. To extend the results to higher rank in Corollary \ref{cor:orbital-to-vertex}, we need to assume $\op{char}\K=0$. Although the explicit solutions are quite different, the combinatorics from \cite{HarRos2016,Har2018} can still be adopted.
Let $R=\K[u_1,u_2,\ldots,u_m]$, and consider two automorphisms $\si_1,\si_2$ of $R$ given by $\si_i(u_j)=u_j-\al_{ji}$ for some $m\times 2$-matrix $\al=(\al_{ij})$.

Let $f$ be a monic irreducible element of $R$, and let $\CO=\{\si_1^k\si_2^l(f)\mid k,l\in\Z\}\in \Irr(R)/\Z^2$ be an orbit of irreducible elements of $R$ with respect to the $\Z^2$-action given by $\si_1,\si_2$.
We think of elements of $\CO$ as the midpoints of the faces of a square lattice, see Figure \ref{fig:lattice}. Applying $\si_i$ corresponds to taking a unit step in the $i$:th direction. Thus taking a half-step in the $i$:th coordinate direction takes us from the midpoint of a face to the (midpoint of) an edge with normal vector in the $i$:th direction. We consider configurations where edges have been assigned a non-negative multiplicity.

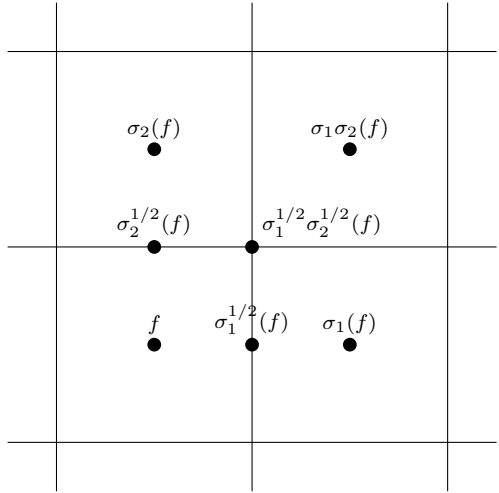
\begin{figure}
\centering
\begin{tikzpicture}[scale=1.3]
\foreach \y in {0,2,4} {
\draw (-.5 cm,\y cm) -- (4.5 cm,\y cm); }
\foreach \x in {0,2,4} {
\draw (\x cm,-.5 cm) -- (\x cm,4.5 cm); }
\draw[thick] (1cm,1cm) node[font=\scriptsize,above] {$f$};
\fill (1,1) circle (2pt);
\fill (2,1) circle (2pt);
\fill (1,2) circle (2pt);
\fill (2,2) circle (2pt);
\fill (1,3) circle (2pt);
\fill (3,3) circle (2pt);
\fill (3,1) circle (2pt);
\draw (2cm,1cm) node[font=\scriptsize,above] {$\si_1^{1/2}(f)$};
\draw (1cm,2cm) node[font=\scriptsize,above] {$\si_2^{1/2}(f)$};
\draw (2cm,2cm) node[font=\scriptsize,above right] {$\si_1^{1/2}\si_2^{1/2}(f)$};
\draw (3cm,1cm) node[font=\scriptsize,above] {$\si_1(f)$};
\draw (1cm,3cm) node[font=\scriptsize,above] {$\si_2(f)$};
\draw (3cm,3cm) node[font=\scriptsize,above] {$\si_1\si_2(f)$};
\end{tikzpicture}
\caption{Part of a square lattice grid with some labels indicated.}
\label{fig:lattice}
\end{figure}

\begin{Definition}
Fix $\CO\in\Irr(R)/\Z^2$. Put $E_i=E_i(\CO)=\si_i^{1/2}(\CO)$ and $V_{12}=V_{12}(\CO)=\si_1^{1/2}\si_2^{1/2}(\CO)$.
Elements of $E_i$ are \emph{edges} of $\CO$ and $V_{12}$ are the \emph{vertices} of $\CO$.

By a \emph{higher spin $6$-vertex configuration} 
$\CL=(\CL_1,\CL_2)$ in $\CO$ we mean a pair of functions $\CL_i:E_i\to \Z_{\ge 0}$ such that for any $v\in V_{12}$:
\begin{align}
\CL_1(\si_2^{1/2}(v))+\CL_2(\si_1^{1/2}(v)) &= \CL_1(\si_2^{-1/2}(v))+\CL_2(\si_1^{-1/2}(v)). \label{eq:6vert}
\end{align}
$\CL$ is \emph{finite} if $\#\{e\in E_i\mid \CL_i(e)\neq 0\}<\infty$ for $i=1,2$.

More generally, we can replace the index set $\{1,2\}$ in the above definition by an arbitrary $2$-subset $\{i,j\}\subseteq\{1,2,\ldots,n\}$, in which case we say that $\CL$ has \emph{index set} $\{i,j\}$.
\end{Definition}

The following theorem shows that there is a bijective correspondence between the set of monic $\CO$-orbital solutions $p=(p_1,p_2)$ to \eqref{eq:binary} and the set of finite higher spin $6$-vertex configurations $\CL$ in $\CO$. 

\begin{Theorem} \label{thm:vertex}
\begin{enumerate}[{\rm (a)}]
\item Given any finite higher spin $6$-vertex configuration $\CL$ in $\CO$, define $p_1,p_2\in R$ by
\begin{equation} \label{eq:sol}
p_i = 
\prod_{e\in E_i} f^{\CL_i(e)}\qquad \forall i\in\{1,2\}
\end{equation}
Then $p=(p_1,p_2)$ is a monic $\CO$-orbital solution to \eqref{eq:binary}.
\item Conversely, given any monic $\CO$-orbital solution $p=(p_1,p_2)$ to 
 \eqref{eq:binary}, then there exists a unique finite cell configuration $\CL$ such that \eqref{eq:sol} holds.
\end{enumerate}
\end{Theorem}

\begin{proof}
Let $i\in\{1,2\}$. Since $R$ is a UFD we can uniquely factor $p_i$ as follows:
\[p_i=\xi_i \cdot \prod_{g\in \Irr(R)} g^{m_i(g)}\]
for some unit $\xi_i$ and some non-negative integers $m_i(g)$. Since $p$ is $\CO$-orbital, each irreducible factor of $p_i$ belongs to $E_i=\si^{1/2}(\CO)$. In other words $m_i(g)=0$ if $g\notin E_i$ and hence we have
\[p_i=\xi_i \cdot \prod_{e\in E_i} e^{m_i(e)}.\]
Define $\CL_i(e)=m_i(e)$ and let $\CL=(\CL_1,\CL_2)$. It remains to show that $\CL$ is a finite higher spin $6$-vertex configuration. Substituting \eqref{eq:sol} into \eqref{eq:binary} we get
\begin{gather}
\prod_{e\in E_i} \si_j^{1/2}(e)^{\CL_i(e)} \prod_{e\in E_j} \si_i^{1/2}(e)^{\CL_j(e)} = \prod_{e\in E_i} \si_j^{-1/2}(e)^{\CL_i(e)} \prod_{e\in E_j} \si_i^{-1/2}(e)^{\CL_j(e)}\\
\prod_{e\in E_k} \si_i^{1/2}\si_j^{1/2}(e)^{\CL_k(e)} \si_i^{-1/2}\si_j^{-1/2}(e)^{\CL_k(e)} = 
\prod_{e\in E_k} \si_i^{1/2}\si_j^{-1/2}(e)^{\CL_k(e)} \si_i^{-1/2}\si_j^{1/2}(e)^{\CL_k(e)} 
\end{gather}
Making appropriate substitutions to write each side as a product over $f\in \CO$ and identifying powers of $f$ in each side, the claim follows.
\end{proof}

\begin{Corollary}\label{cor:orbital-to-vertex}
Suppose $\op{char}\K=0$.
Let $m$ and $n$ be positive integers.
Let $R=\K[u_1,u_2,\ldots,u_m]$, $(\al_{ij})$ be an $m\times n$-matrix, $\si_i(u_j)=u_j-\al_{ji}$, $p=(p_1,p_2,\ldots,p_n)$ be a monic $\CO$-orbital solution to \eqref{eq:binary}. Then there exists two distinct indices $i,j\in\{1,2,\ldots,n\}$ and a finite higher spin $6$-vertex configuration $\CL$ with index set $\{i,j\}$ such that
\[p_i=\prod_{e\in E_i} e^{\CL_i(e)},
\qquad p_j = \prod_{e\in E_j} f^{\CL_j(e)},
\qquad \text{$p_k=1$ for $k\notin\{i,j\}$.}
\]
\end{Corollary}
\begin{proof}
Apply Corollary \ref{cor:nontriv} and Theorem \ref{thm:vertex}.
\end{proof}

We summarize the results of the paper in the following statement.

\begin{MainTheorem}\label{thm:summary}
Let $m$ and $n$ be positive integers, $n\ge 2$. Let $\K$ be a field of characteristic zero.
Let $R=\K[u_1,u_2,\ldots,u_m]$ and let $(\al_{ij})\in M_{m\times n}(\K)$.

Let $(p_1,p_2,\ldots,p_n)\in R^n$ be an $n$-tuple of monic polynomials satisfying the following system of equations:
\begin{equation}
\label{eq:binary2}
p_i(u-\tfrac{\al_j}{2})\cdot p_j(u-\tfrac{\al_i}{2}) 
=p_i(u+\tfrac{\al_j}{2})\cdot p_j(u+\tfrac{\al_i}{2}) \quad \forall i\neq j,
\end{equation}
where $u=(u_1,u_2,\ldots,u_m)$ and $\al_i=(\al_{1i},\al_{2i},\ldots,\al_{mi})$.
Then there exist
\begin{enumerate}[{\rm (i)}]
\item 
 a finite set of (distinct) orbits $\CO^{(1)}, \CO^{(2)}, \ldots, \CO^{(d)} \in \Irr(R)/\Z^n$ in the set of monic irreducible elements $\Irr(R)$ with respect to the action of $\Z^n$ by $\K$-algebra automorphisms on $R$ determined by
$\boldsymbol{x}.u_j = u_j - \sum_{i=1}^n x_i\al_{ji}$ for all $\boldsymbol{x}=\sum_{i=1}^n x_k\boldsymbol{e}_k\in\Z^n$;
\item
 for each orbit $\CO^{(a)}$, a $2$-subset $\{i_a,j_a\}\subseteq\{1,2,\ldots,n\}$ such that for all $k\in\{1,2,\ldots,n\}\setminus\{i_a,j_a\}$ we have $\boldsymbol{e}_k.q=q$ for all $q\in\CO^{(a)}$;
\item 
 for each orbit $\CO^{(a)}$, a finite higher spin $6$-vertex configuration $\CL^{(a)}$, with index set $\{i_a,j_a\}$, in $\CO^{(a)}$ such that
\begin{subequations}\label{eq:p-definitions}
\begin{equation}
p_k = p_k^{(1)} p_k^{(2)}\cdots p_k^{(d)} \qquad \forall k\in\{1,2,\ldots,n\},
\end{equation}
\begin{equation}
p_k^{(a)} = \begin{cases}
 \displaystyle \prod_{f\in E_k} f^{\CL^{(a)}_k(f)}, & \text{if $k\in\{i_a,j_a\}$,} \\
1, & \text{otherwise.} 
\end{cases}
\end{equation}
\end{subequations}
\end{enumerate}
Conversely, given orbits $\CO^{(a)}$, $2$-subsets $\{i_a,j_a\}$ and vertex configurations $\CL^{(a)}$ as in {\rm (i)--(iii)}, then $(p_1,p_2,\ldots,p_n)$ defined by \eqref{eq:p-definitions}, satisfies equations \eqref{eq:binary2}.

Moreover, the following relations hold (non-vacuously when $n\ge 3$):
\begin{equation}
\label{eq:thm-ternary}
p_k(u-\tfrac{\al_i}{2}-\tfrac{\al_j}{2})p_k(u+\tfrac{\al_i}{2}+\tfrac{\al_j}{2})=p_k(u-\tfrac{\al_i}{2}+\tfrac{\al_j}{2})p_k(u+\tfrac{\al_i}{2}-\tfrac{\al_j}{2})\quad \forall i\neq j\neq k\neq i.
\end{equation}
where $u=(u_1,u_2,\ldots,u_m)$ and $\al_i=(\al_{1i},\al_{2i},\ldots,\al_{mi})$.
\end{MainTheorem}

\begin{proof}
By combining Theorem \ref{theorem:integral-product} and Corollary \ref{cor:orbital-to-vertex}. The last claim then follows from Corollary \ref{cor:binary-implies-ternary}.
\end{proof}

\subsection{Examples}

\begin{Example} \label{ex:example1}
Let $(m,n)=(2,3)$, $\al=\left[\begin{smallmatrix}-1&1&0\\0&-1&1\end{smallmatrix}\right]$, $R=\C[u_1,u_2]$ and $\si_i(u_j)=u_j-\al_{ji}$.
Then 
\[p=(p_1,p_2,p_3)=\big(u_1-\tfrac{1}{2},\,(u_1+\tfrac{1}{2})(u_2-\tfrac{1}{2}),\,u_2+\tfrac{1}{2}\big)\]
is a solution to the consistency equations \eqref{eq:consistency-eqs} related to the Lie algebra $\mathfrak{gl}_3$, see \cite{Ser2001}. It is a monic solution and its factorization into orbital solutions is  $p=p^{(1)}p^{(2)}$ where $p^{(i)}$ is $(\Z^3\cdot u_i)$-orbital and given by
\begin{equation}
p^{(1)}=(u_1-\tfrac{1}{2},\,u_1+\tfrac{1}{2},\,1),\qquad
p^{(2)}= (1,\,u_2-\tfrac{1}{2},\,u_2+\tfrac{1}{2}).
\end{equation}
Also note that $\{i_1,j_1\}=\{1,2\}$ and $\{i_2,j_2\}=\{2,3\}$.
\end{Example}

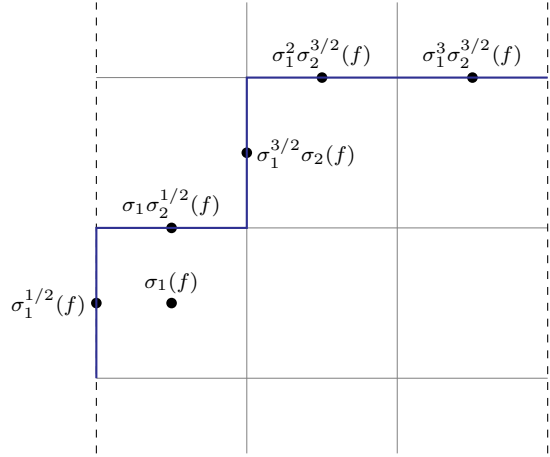
\begin{figure}[h!]
\centering
\begin{tikzpicture}[scale=2]
\foreach \y in {-1,0,...,1} {
\draw[help lines] (0 cm,\y cm) -- (3 cm,\y cm); }
\foreach \x in {1,...,2} {
\draw[help lines] (\x cm,-1.5 cm) -- (\x cm,1.5 cm); }
\draw[dashed] (0,-1.5 cm) -- (0,1.5cm);
\draw[dashed] (3,-1.5 cm) -- (3,1.5cm);
\fill (.5,-.5) circle (1pt);
\draw (.5cm,-.5cm) node[font=\scriptsize,above] {$\si_1(f)$};

\fill (0,-.5) circle (1pt) node[font=\scriptsize,left] {$\si_1^{1/2}(f)$};

\fill (.5,0) circle (1pt) node[font=\scriptsize,above] {$\si_1\si_2^{1/2}(f)$};

\fill (1,.5) circle (1pt) node[font=\scriptsize,right] {$\si_1^{3/2}\si_2(f)$};

\fill (1.5,1) circle (1pt) node[font=\scriptsize,above] {$\si_1^2\si_2^{3/2}(f)$};

\fill (2.5,1) circle (1pt) node[font=\scriptsize,above] {$\si_1^3\si_2^{3/2}(f)$};

\draw[thick,Blue] (0,-1) -- (0,0) -- (1,0) -- (1,1) -- (2,1)--(3,1);
\end{tikzpicture}
\caption{Diagram depicting a lattice configuration $\CL$ corresponding to the solution from Example \ref{ex:example2}. Edges $e$ for which $\CL_i(e)=1$ are blue (solid).}
\label{fig:example}
\end{figure}

\begin{Example} \label{ex:example2}

Let $(m,n)=(3,4)$, $R=\Q[u_1,u_2,u_3]$, 

$f(u_1,u_2,u_3)=(u_2+u_3)^2-(u_1^3-u_1+1)$, \quad $\al=\left[\begin{matrix} 2& -3 & 0 & 0 \\ 4 & -5 & 1 & -3\\ -2 & 2 & -1 & 3 \end{matrix}\right]$.

Notice that we have $\si_3 f=\si_4 f=f=\si_1^3\si_2^2 f$. We can then define a higher spin $6$-vertex configuration, with index set $\{1,2\}$, given by

\begin{align*}
\CL_1(e) &=
\begin{cases}
1, & 
\text{if $e\in\big\{\si_1^{1/2}(f),\,\si_1^{3/2}\si_2(f)\big\}$}\\
0, & \text{otherwise}
\end{cases} \\
\CL_2(e) &=
\begin{cases}
1, & 
\text{if $e\in\big\{\si_1\si_2^{1/2}(f),\,\si_1^2\si_2^{3/2}(f),\, \si_1^3\si_2^{3/2}(f)\big\}$}\\
0, & \text{otherwise}
\end{cases}
\end{align*}

From Figure \ref{fig:example}, it is clear that this configuration satisfies \eqref{eq:6vert} and hence corresponds to an orbital solution, which can be written as $p=(p_1,p_2,p_3,p_4)$, where
\begin{align*}
 p_1&=\si_1^{1/2}(f)\si_1^{3/2}\si_2(f)\\
&=f\big(u_1-\tfrac{1}{2}\al_{11},\, u_2-\tfrac{1}{2}\al_{21},\, u_3-\tfrac{1}{2}\al_{31}\big)\cdot f\big(u_1-\tfrac{3}{2}\al_{11}-\al_{12},\, u_2-\tfrac{3}{2}\al_{21}-\al_{22},\, u_3-\tfrac{3}{2}\al_{31}-\al_{32}\big)\\
&=\left((u_2+u_3-1)^2-((u_1-1)^3-u_1+2)\right)\cdot \left((u_2+u_3-1)^2-((u_1+1)^3-u_1)\right)
\end{align*}
\begin{align*}
 p_2 &= \si_1\si_2^{1/2}(f)\cdot \si_1^2\si_2^{3/2}(f)\cdot \si_1^3\si_2^{3/2}(f)\\
&=
\left((u_2+u_3+\tfrac{3}{2})^2-((u_1+\tfrac{3}{2})^3-u_1-\tfrac{1}{2})\right)\cdot \left((u_2+u_3+\tfrac{1}{2})^2-((u_1+\tfrac{1}{2})^3-u_1-\tfrac{3}{2})\right)\\
&\phantom{=}\cdot \left((u_2+u_3-\tfrac{5}{2})^2-((u_1-\tfrac{5}{2})^3-u_1+\tfrac{5}{2})\right)
\end{align*}
$$p_3=1,\qquad p_4=1.$$
\end{Example}

\section{Irreducible factorization of multiquiver solutions}

In this section we show how the solutions to the consistency equations obtained in \cite{HarSer2016} fit into our classification scheme.

\subsection{Input matrix}
Let $\beta=(\beta_{ij})$ be an $m\times n$-matrix satisfying the following conditions:
\begin{align}
& \text{$\beta_{ij}\in\Z$ for all $i,j$,} \label{eq:beta-1} \\
& \text{$\beta_{ij}\beta_{ik}\le 0$ for all $i,j,k$ with $j\neq k$.} \label{eq:beta-2}
\end{align}
Condition \eqref{eq:beta-2} is equivalent to
\begin{equation} \label{eq:beta-3}
\text{Every row of $\beta$ contains at most one positive and at most one negative integer.}
\end{equation}
Moreover $\beta$ can be interpreted as the incidence graph of certain multiquivers, see \cite{HarSer2016} for details.

\subsection{Solution attached to $\beta$}
Let $R=\K[u_1,u_2,\ldots,u_m]$ and $\sigma=(\sigma_1,\sigma_2,\ldots,\sigma_n)$ be given by $\sigma_i(u_j)=u_j-\beta_{ji}$ and define $p=(p_1,p_2,\ldots,p_n)\in R^n$ by

\begin{equation}
p_i(u_1,u_2,\ldots,u_m)=\prod_{j=1}^m p_{ji}(u_j),
\end{equation}
where
\begin{equation}
p_{ji}(u_j)=\begin{cases}
u_j(u_j+1)\cdots (u_j+\beta_{ji}-1) & \text{if $\beta_{ji}>0$},\\
1 &\text{if $\beta_{ji}=0$},\\
(u_j-1)(u_j-2)\cdots (u_j-|\beta_{ji}|) & \text{if $\beta_{ji}<0$}.
\end{cases}
\end{equation}

\begin{Theorem}[{\cite{HarSer2016}}]
$\sigma$ and $p$ satisfy the TGWA consistency relations from \cite{FutHar2012a}. That is,
\begin{subequations}\label{eq:TGWA-consistency-nonsymmetric}
\begin{align}
\si_i\si_j(p_ip_j)&=\si_i(p_i)\si_j(p_j) \quad \forall i\neq j,
\label{eq:non-symmetric-1} \\
\si_i\si_k(p_j)p_j &= \si_i(p_j)\si_k(p_j) \quad \forall i\neq j\neq k\neq i.
\label{eq:non-symmetric-2}
\end{align}
\end{subequations}
\end{Theorem}

\subsection{Symmetrized version}\label{sec:symmetrized}
We write \eqref {eq:TGWA-consistency-nonsymmetric}  in the symmetric form. Applying $\si_i^{-1/2}\si_j^{-1/2}$ to equation \eqref{eq:non-symmetric-1} and putting $\tilde p_i=\si_i^{1/2}(p_i)$ gives
\begin{equation} \label{eq:symm-binary}
\si_j^{1/2}(\tilde p_i)\si_i^{1/2}(\tilde p_j)=\si_j^{-1/2}(\tilde p_i)\si_i^{-1/2}(\tilde p_j)
\end{equation}
That is, $\tilde p:=(\tilde p_1,\tilde p_2,\ldots,\tilde p_n)$ satisfies the symmetrized binary TGWA consistency equation.
Similarly applying $\si_i^{-1/2}\si_k^{-1/2}\si_j^{1/2}$ to equation \eqref{eq:non-symmetric-2} gives the equivalent relation
\begin{equation}\label{eq:symm-ternary}
\si_i^{1/2}\si_k^{1/2}(\tilde p_j)\si_i^{-1/2}\si_k^{-1/2}(\tilde p_j) = \si_i^{1/2}\si_k^{-1/2}(\tilde p_j)\si_i^{-1/2}\si_k^{1/2}(\tilde p_j)
\end{equation}
which shows that $\tilde p$ also satisfies the symmetrized ternary TGWA consistency equation.
In \cite{HarSer2016}, relation \eqref{eq:non-symmetric-2} was proved separately.
But in fact, as we showed in Corollary \ref{cor:binary-implies-ternary}, equation \eqref{eq:symm-ternary} actually follows from \eqref{eq:symm-binary}, hence \eqref{eq:non-symmetric-2} follows from \eqref{eq:non-symmetric-1}.

Explicitly, we have
\begin{align*}
\tilde p_i &= \sigma_i^{1/2}(p_i) = p_i\big(u_1-\frac{\beta_{1i}}{2},u_2-\frac{\beta_{2i}}{2},\ldots,u_m-\frac{\beta_{mi}}{2}\big)\\
&= \prod_{j=1}^m p_{ji}(u_j-\frac{1}{2}\beta_{ji})
\end{align*}
To write a closed formula we also make an overall shift by $1/2$.
Put $q_{ji}(u_j)=p_{ji}(u_j-\frac{1}{2}\beta_{ji}+\frac{1}{2})$. Then 
\begin{equation}
 q_{ji}(u_j)=
\begin{cases}
(u_j-\frac{1}{2}\beta_{ji}+\frac{1}{2})(u_j-\frac{1}{2}\beta_{ji}+\frac{3}{2})\cdots (u_j+\frac{1}{2}\beta_{ji}-\frac{1}{2}),&\text{if $\beta_{ji}>0$}\\
1,&\text{if $\beta_{ji}=0$}\\
(u_j-\frac{1}{2}|\beta_{ji}|+\frac{1}{2})(u_j-\frac{1}{2}|\beta_{ji}|+\frac{3}{2})\cdots (u_j+\frac{1}{2}|\beta_{ji}|-\frac{1}{2}),&\text{if $\beta_{ji}<0$}
\end{cases}
\end{equation}
We can summarize these observations as follows.

\begin{Proposition}
Let $\beta$ be an $m\times n$ matrix satisfying conditions \eqref{eq:beta-1} and \eqref{eq:beta-2}. Put 
\begin{equation}
\tilde \beta_{ji}=\frac{1}{2}(|\beta_{ji}|-1)
\end{equation}
Define polynomials for $(i,j)\in\iv{1}{n}\times\iv{1}{m}$:
\begin{equation}
q_{ji}(u_j)= (u_j-\tilde \beta_{ji})(u_j-\tilde\beta_{ji}+1)\cdots (u_j+\tilde\beta_{ji})
\end{equation}
then for all $j\in\iv{1}{m}$ the $n$-tuple
\[q_j=(q_{j1}(u_j),q_{j2}(u_j),\ldots,q_{jn}(u_j))\]
is a solution to the consistency equations \eqref{eq:consistency-eqs}.
\end{Proposition}

\begin{proof}
It suffices to observe that if $j\neq j'$, then 
\[(q_{j1}(u_j),q_{j2}(u_j),\ldots,q_{jn}(u_j))\]
and
\[(q_{j'1}(u_{j'}),q_{j'2}(u_{j'}),\ldots,q_{j'n}(u_{j'}))\]
belong to different orbits with respect to the action of $\Z^n$ via $\si$. More precisely, given $j\neq j'$ in $\iv{1}{m}$ and $i,i'\in\iv{1}{n}$ then there are no integers $k_s$ such that some irreducible factor of $\si_1^{k_1}\si_2^{k_2}\cdots\si_n^{k_n}(q_{ji})$ divides $q_{j'i'}$. This is obvious since the former is a univariate polynomial in $u_j$ and the latter a univariate polynomial in $u_{j'}$.
\end{proof}

\subsection{Factorization into irreducibles}
By \cite[Prop.~6.2]{HarRos2016}, for each $j\in\iv{1}{m}$ the solution $q_j$ factors into a product of  $\gcd(\beta_{ji_1},\beta_{ji_2})$ solutions, where $\{i_1,i_2\}=\{i\in\iv{1}{n}\mid \beta_{ji}\neq 0\}$ (assume without loss of generality that all rows in $\beta$ are nonzero).
So in total the solution from \cite{HarSer2016}, corresponding to a matrix $\beta$, factors into
\[
\prod_{j=1}^m \gcd(\beta_{j1}, \beta_{j2},\ldots,\beta_{jn})
\]
orbits. Each orbit factor is actually irreducible and corresponds to a (unique) zero area generalized Dyck path. Explicitly, let $\ga_j = \gcd(\beta_{j1},\beta_{j2},\ldots,\beta_{jn})$. We have
\[
(q_{j1}(u_j),q_{j2}(u_j),\ldots, q_{jn}(u_j))=
(1,\ldots,q_{ji_1}(u_j),\ldots,q_{ji_2}(u_j),\ldots,1)
\]
i.e. all entries are $1$ except possibly $2$ places. Then, following \cite{HarRos2016}, we can further factor
\[
(q_{ji_1}(u_j),q_{ji_2}(u_j))=
\prod_{k=0}^{\ga_j-1} (q_{ji_1}^{(k)}(u_j),q_{ji_2}^{(k)}(u_j))
\]
where $q_{ji}^{(k)}(u_j)$ is the product of all factors of the form $(u_j-l)$ where $l\equiv k\mod {\ga_j}$. The factors $q_{ij}^{(k)}(u_j)$ were expressed using zero area generalized Dyck paths in \cite{HarRos2016}.

\bibliographystyle{siam} 

\end{document}